\documentclass[a4paper,10pt,reqno]{amsart}

\usepackage[utf8]{inputenc}
\usepackage[T1]{fontenc}
\usepackage{amsthm}
\usepackage{amsmath}
\usepackage{amssymb}
\usepackage[inline]{enumitem}
\usepackage{comment}
\usepackage{hyperref}
\usepackage{fancyhdr}
\usepackage{mathrsfs}
\usepackage{stmaryrd}
\usepackage[normalem]{ulem}
\usepackage{xcolor}
\usepackage{nicefrac}
\usepackage{tikz}

\theoremstyle{definition}

\newtheorem{theorem}{Theorem}[section]
\newtheorem{lemma}[theorem]{Lemma}
\newtheorem{proposition}[theorem]{Proposition}
\newtheorem{corollary}[theorem]{Corollary}

\theoremstyle{definition}

\newtheorem{remark}[theorem]{Remark}

\providecommand\Mon{{\mathcal{M}{\rm on}}}

\DeclareMathOperator{\mdeg}{\mathsf {m-deg}}

\title{On atoms and their density in the monoid of monomial ideals}

\author{Nikola Bogdanovic}
\address{(N. Bogdanovic) Department of Mathematics and Scientific Computing, University of Graz | Heinrichstrasse 36/III, 8010 Graz, Austria}
\email{nikola.bogdanovic@uni-graz.at}

\author{Laura Cossu}
\address{(L. Cossu) Department of Mathematics and Computer Science, University of Cagliari | Via Ospedale 72, 09124 Cagliari, Italy}
\email{laura.cossu3@unica.it}



\thanks{This research was funded in part by the Austrian Science Fund (FWF) through the project 10.55776/DOC183. N.~B.~acknowledges financial support from the Department of Economy, Tourism, Science and Research of the Federal Province of Styria (Austria), as well as the hospitality of the Department of Mathematics and Computer Science of the University of Cagliari, where he spent the winter semester 2025/2026 and where this work was initiated and took shape.  L.~C.~acknowledges support from the GNSAGA research group of the Italian National Institute of High Mathematics (INdAM)}

\begin{document}
\begin{abstract}
We study the structure and distribution of atoms in the monoid $\Mon(R)$ of nonzero monomial ideals of the polynomial ring $R=K[X_1,\ldots, X_N]$, with $N\ge 2$. We characterize the atoms with at most four minimal generators and introduce the density $d(N,\mu)$ of atoms among monomial ideals with $\mu$ minimal generators. For $\mu\le 4$, we prove the existence of this density and determine its asymptotic behavior as $N\to\infty$. In the bivariate case, we also obtain a uniform lower bound for the density of atoms in a divisor-closed submonoid of $\Mon(R)$ generated by equigenerated ideals.
\end{abstract}
\maketitle

\section{Introduction}
The multiplicative structure of ideals in integral domains is a central object of multiplicative ideal theory. Given a domain, or more generally a cancellative commutative monoid, one can endow suitable families of ideals with a multiplication induced by the ordinary product of ideals. The resulting monoids of ideals (or {\it ideal monoids}) encode substantial information about the underlying algebraic structure and have been studied extensively in the language of ideal systems, star operations, and semigroup theory; see, for instance, \cite{Hal-Ko98,Fo-Ho-Lu13,Rein12}. Recall that a commutative monoid $H$, i.e., a commutative semigroup with identity, is {\it cancellative} if $ab=ac$ implies $b=c$ for all $a,b,c\in H$.

The best-understood examples of ideal monoids arise by restricting attention to classes of ideals for which multiplication has good cancellation and finiteness properties. It is well known that nonzero ideals of a Dedekind domain form a free abelian monoid, with the nonzero prime ideals as a free basis. More generally, for a Krull domain, the monoid of invertible ideals provides a basic example of a Krull monoid, although it is not factorial in general \cite[Section 4]{Ge-Kh22}. 

From the point of view of factorization theory, the focus shifts from classical ideal decompositions to the arithmetic of factorizations into multiplicatively irreducible elements (or {\it atoms}) of the ideal monoid. The main objects of study are then these atoms and the invariants measuring the non-uniqueness of factorization, such as sets of lengths, unions of sets of lengths, and elasticities. This perspective has led to an extensive investigation of the arithmetic of ideal monoids in both cancellative and non-cancellative settings. We refer to the monograph \cite{Ge-HK} for the general background in factorization theory and to \cite{Fan17, Ge-Re19, Ge-Kh22} for results specifically concerned with monoids of ideals.

A substantial part of this literature deals with monoids of divisorial and invertible ideals, a setting in which divisor-theoretic methods and transfer homomorphisms frequently yield precise descriptions of the arithmetic. For example, arithmetical finiteness results for divisorial ideal monoids of weakly Krull Mori domains, including results on sets of lengths and monotone catenary degrees, were established in
\cite{Ge-Re19}. The arithmetic of monoids of $r$-invertible $r$-ideals in Krull and weakly Krull Mori domains, for suitable ideal systems $r$, was studied further in \cite{Ge-Kh22}. Radical ideal factorizations within the framework of multiplicative lattices and finitary ideal systems have also been
considered in \cite{Ol-Re19} and \cite{Ol-Re20}.

Much less is known about the monoid $\mathcal I(R)$ of all nonzero ideals of an integral domain $R$, endowed with ordinary ideal multiplication. This monoid is typically non-cancellative, so many of the standard tools of classical factorization theory are no longer directly available. Factorization theory has also been developed in non-cancellative settings, in particular through the study of commutative rings with zero divisors and the various notions of irreducibility and atomicity arising in this context; see \cite{AnVL96,AnVL97}. Factorization phenomena for ideals of arbitrary commutative rings were subsequently studied from a related perspective in \cite{An-Ju-Mo19}, with particular emphasis on multiplicative irreducibility and the corresponding decomposition properties. Other works in the area include \cite{Jue13, GJNR24}. A more systematic approach to factorization in the broader non-cancellative context is based on the notion of {\it unit-cancellativity}, a weakening of cancellativity. A commutative monoid $H$ is called {\it unit-cancellative} if $a=ab$ implies that $b$ is a unit of $H$, for all $a,b\in H$. A substantial part of the classical theory, together with many of its arithmetic invariants, carries over naturally to this setting; see \cite{Fan17,Ge-Zh20}. Nevertheless, the structure of factorizations in $\mathcal I(R)$ remains poorly understood even for natural classes of domains $R$.

An important case is that of a multivariate polynomial ring $R=K[X_1,\ldots,X_N]$, where $K$ is a field and $N\geq 2$. The study of the arithmetic of $\mathcal I(R)$ was initiated in \cite{Ge-Kh22}, where it was shown that this monoid shares striking factorization-theoretic features with Krull monoids having infinite class group and prime divisors in all classes, despite not being transfer Krull. Further families of atoms and sets of lengths in $\mathcal I(R)$ were studied in \cite{BCK}. These results also highlighted the submonoid $\Mon(R)$ of nonzero monomial ideals as a natural and particularly concrete object of study. In particular, it was shown in \cite[Section 4]{BCK} that $\Mon(R)$ is neither locally finitely generated nor transfer Krull, while exhibiting a rich arithmetic of factorizations, including full elasticity and maximal unions of sets of lengths. Moreover, families of atoms of $\Mon(R)$ were constructed, and sets of lengths were computed for certain classes of ideals in this monoid. A related factorization problem for integrally closed monomial ideals, endowed with the product obtained by taking the integral closure of the ordinary ideal product, was recently studied in \cite{Lew25} using the geometry of Newton polyhedra.

The aim of the present paper is to investigate atoms of $\Mon(R)$, endowed with the usual ideal multiplication, from both a structural and an asymptotic point of view. More precisely, we address the following two questions:

\vspace{.1cm}
\begin{enumerate*}[label=\textup{(\arabic{*})}]
    \item Can atoms of $\Mon(R)$ having a small prescribed number of minimal generators be characterized explicitly?
\end{enumerate*}

\vspace{.1cm}
\begin{enumerate*}[label=\textup{(\arabic{*})}, resume]
    \item How frequent are atoms among monomial ideals of bounded combinatorial complexity?
\end{enumerate*}

\vspace{.1cm}
The two questions are closely connected, since explicit criteria for reducibility make it possible not only to recognize individual atoms, but also to count them and study their asymptotic distribution.

\subsection*{Plan of the paper} In Section~2, we fix the notation and collect some preliminary results. In Section~3, we establish some basic factorization properties of principal monomial ideals, prove that every finite interval $[m,n]$ with $2\le m\le n$ occurs as a set of lengths in $\Mon(R)$, and characterize the atoms with at most four minimal generators. In Section~4, we introduce the density $d(N,\mu)$ of atoms among $N$-variate monomial ideals with $\mu$ minimal generators, and prove its existence and asymptotic behavior for $\mu\le 4$. Finally, in Section~5 we study the divisor-closed submonoid $\mathcal E(R)$ of $\Mon(R)$ generated by the equigenerated monomial ideals whose minimal generators have greatest common divisor $1$. In the bivariate case, we derive a uniform lower bound for the density of atoms in $\mathcal E(R)$.

\section{Basic notation and preliminaries}

Throughout the paper, $\mathbb{N}$ and $\mathbb{Z}$ denote the set of nonnegative integers and the ring of integers, respectively. For $m\in \mathbb N$ and $n\in\mathbb N \cup \{\infty\}$ we let $[ m, n] := \allowbreak \{x\in \mathbb{N} \colon m\leq x\leq n\}$ be the discrete interval from $m$ to $n$.

\medskip

Let $R=K[X_1,\dots,X_N]$, where $K$ is a field and $N\ge 2$. For $a,b\in \mathbb N^N$, define $a\wedge b$ by setting \[(a\wedge b)(i) = \min\{a(i), b(i)\}\quad \text{ for all }i\in [1,N].\]

A monomial $X_1^{a(1)}\dots X_N^{a(N)}\in R$ can be written more compactly as $\underline X^a$, with $a=(a(1),\dots,a(N))\in \mathbb N^N\subseteq \mathbb Z^N$. For $a_1,\ldots, a_k\in\mathbb N^N$, we denote by
\[\langle \underline{X}^{a_1},\ldots,\underline{X}^{a_k}\rangle\]
the monomial ideal generated by the monomials $\underline{X}^{a_1},\ldots,\underline{X}^{a_k}$.
We will often regard $\mathbb Z^N$ as a partially ordered group, where the operation is component-wise addition and the order is the product order of the usual total order on $\mathbb Z$, i.e., $a\le b$ if and only if $a(i)\le b(i)$ for every $i\in[1,N]$. We write $a<b$ if $a\leq b$ and $a\neq b$. Denote by $\underline 0$ the identity element of $\mathbb Z^N$, i.e., the zero vector. Then $\mathbb N^N = \{a\in\mathbb Z^N\colon a\ge \underline 0\}$ is the positive cone of $\mathbb Z^N$. We write $a\perp b$ to indicate that $a,b\in\mathbb Z^N$ are incomparable with respect to the order $\le$, i.e., neither $a\le b$ nor $b\le a$ holds.

The divisibility partial order on the set $\mathcal M(R)$ of monomials of $R$, plays a crucial role in the study of monomial ideals of $R$. Clearly, the map \[\iota\colon (\mathbb N^N,\le) \to (\mathcal M(R),\vert),\] 
sending $a$ to $\underline X^a$, is an isomorphism of partially ordered monoids, where the operation between monomials is the product. For $a_1,\dots, a_k\in\mathbb N^N$,
we have
\[ \underline X^{\bigwedge_{i=1}^k a_i} = \gcd{(\underline X^{a_1}, \dots, \underline X^{a_k})}.  \]

\medskip

Let $\Mon(R)$ be the (unit-cancellative, reduced) monoid of nonzero monomial ideals of $R$. For $I\in\Mon(R)$, we denote by $G(I)$ the unique minimal set of monomial generators of $I$, and by $\mu(I)$ its cardinality. By Dickson's Lemma $G(I)$ is finite and hence $\mu(I)$ is a positive integer. In particular, $I=R$ if and only if $G(I)=\{1\}$ (and so $\mu(I)=1$). Thus, every $I\in \Mon(R)$ with $\mu(I)\ge 2$ is different from $R$. Throughout, whenever we refer to a \lq\lq monomial ideal\rq\rq, we tacitly assume that it is nonzero, i.e., an element of $\Mon(R)$.

The following classical lemma (see, e.g., \cite[Cor.~1.8]{En-Her} for a proof), will be used repeatedly without further comment.
\begin{lemma}\label{lem: mon_gen}
Let $I\in\Mon(R)$, and $\mathcal G$ be a set of monomials in $I$. Then $\mathcal G$ generates $I$ if and only if for each monomial $v\in I$ there exists $u\in\mathcal G$ such that $u\mid v$.
\end{lemma}

We will also need the following elementary multiplicative property of the greatest common divisor of the minimal generators of a monomial ideal.
\begin{lemma}\label{lem: gcd IJ}
    For any $I,J\in\Mon(R)$, we have
    \[ \gcd(G(IJ)) = \gcd(G(I))\cdot\gcd(G(J)). \]
\end{lemma}
\begin{proof}
    Write $I=\gcd(G(I))\cdot I'$ and $J=\gcd(G(J))\cdot J'$, where $I',J'\in\Mon(R)$ and $\gcd(G(I'))=\gcd(G(J'))=1$. Clearly, we have $IJ = \gcd(G(I))\cdot \gcd(G(J))\cdot I'J'$ and $\gcd(G(IJ)) = \gcd(G(I))\cdot \gcd(G(J))\cdot \gcd(G(I'J'))$, so to conclude it suffices to show that $\gcd(G(I'J'))=1$. To achieve this, we prove that, for all $k\in [1,N]$, there is a minimal generator of $I'J'$ which is not divisible by $X_k$. Indeed, as $\gcd(G(I'))=1$ and $\gcd(G(J'))=1$, there are $f_k\in G(I')$ and $g_k\in G(J')$ such that $X_k$ divides neither $f_k$ nor $g_k$. As $f_kg_k\in I'J'$, there exists $h_k\in G(I'J')$ dividing $f_kg_k$, whence $X_k$ does not divide $h_k$.
\end{proof}

For a monomial ideal $I\in \Mon(R)$, we define its {\it min-degree} by:
\[\mdeg{I}=\min\{\deg(u) \; \colon \; u\in G(I)\}.\]
The min-degree is additive with respect to ideal multiplication: for all $I,J\in\Mon(R)$,
\[\mdeg(IJ)=\mdeg(I)+\mdeg(J).\]
Indeed, the ideal $IJ$ is generated by the monomials $uv$, with $u\in G(I)$ and $v\in G(J)$. Hence every element $w\in G(IJ)$ is divisible by some such product $uv$, and therefore $\deg(w)\geq \deg(u)+\deg(v)\geq \mdeg(I)+\mdeg(J)$. Thus,\[\mdeg(IJ)\geq \mdeg(I)+\mdeg(J).\]
Conversely, let $u\in G(I)$ and $v\in G(J)$ have degrees $\mdeg(I)$ and $\mdeg(J)$, respectively. Since $uv\in IJ$, some $w\in G(IJ)$ divides $uv$, and so
\[\mdeg(IJ)\leq \deg(w)\leq \deg(uv)=\mdeg(I)+\mdeg(J).\]

\section{Factorizations, lengths, and atoms in $\Mon(R)$}

Let $R=K[X_1,\ldots, X_N]$, with $K$ a field and $N\ge 2$. In this section, we study atoms of $\Mon(R)$ with a small number of minimal generators. More precisely, we give a complete characterization of atoms $I\in\Mon(R)$ with $\mu(I)\le 4$.

Before turning to this classification, we establish some basic factorization properties of principal monomial ideals and derive from them a realization result for sets of lengths. 

Recall that a \textit{factorization} of $I\in\Mon(R)$ is a formal product $z=A_1\cdots A_k$, such that \[I=A_1\cdots A_k,\]
where $A_1,\ldots,A_k$ are atoms of $\Mon(R)$, that is, proper monomial ideals of $R$ that cannot be written as a product of two proper monomial ideals. The integer $k$ is called the \textit{length} of $z$.
We denote by $\mathsf Z_{\Mon(R)}(I)$ the set of factorizations of $I$, where factorizations differing only by a permutation of the factors are identified. We also allow the empty factorization, which is the unique factorization of $R$ and has length $0$. The \textit{set of lengths} of $I$ is then defined by
\[
\mathsf L_{\Mon(R)}(I):=\{k\in\mathbb N : I \text{ admits a factorization of length } k\}.
\] 
It was proved in \cite[Theorem 4.7]{BCK} that $\Mon(R)$ is a $\mathsf{BF}$-monoid; equivalently, every element of $\Mon(R)$ admits a factorization and its set of lengths is finite.

Finally, recall that a monomial ideal $I$ is called a \textit{strong atom} (or an \textit{absolutely irreducible}) of $\Mon(R)$ if it is an atom and
\[\mathsf Z_{\Mon(R)}(I^n)=\{I^n\} \qquad\text{for every }n\ge 1.\]

\begin{lemma}\label{lem: princ div closed}
    Let $\underline X^a\in \mathcal M(R)$ with $a>\underline 0$ and suppose $\langle \underline X^a\rangle = JK$ for some $J,K\in\Mon(R)\setminus\{R\}$. Then both $J$ and $K$ are generated by a single non-constant monomial. In other words, the submonoid of $\Mon(R)$ consisting of principal monomial ideals is divisor-closed. In particular:
    \begin{enumerate}[label=\textup{(\arabic{*})}, mode=unboxed]
    \item\label{lem: princ div closed 1} $\mathsf Z_{\Mon(R)}(\langle \underline X^a\rangle) = \{ \langle X_1\rangle^{a(1)}\dots\langle X_N\rangle^{a(N)}\}$;
    \item\label{lem: princ div closed 2} for any $i\in [1,N]$, the ideal $\langle X_i\rangle$ is a strong atom of $\Mon(R)$.
    \end{enumerate}
\end{lemma}

\begin{proof}
    Let $G(J) = \{\underline X^{\beta_1}, \dots, \underline X^{\beta_{\mu(J)}}\}$, $G(K) = \{ \underline X^{\gamma_1}, \dots, \underline X^{\gamma_{\mu(K)}}\}$. We claim that $\mu(J)=\mu(K)=1$. Without loss of generality, we can assume $a=\beta_1+\gamma_1$. Suppose $\mu(K) \ge 2$: since $\underline X^{\beta_1+\gamma_2}\in \langle \underline X^a\rangle$, necessarily $a\le \beta_1+\gamma_2$. Combined with the previous equality, we obtain $\gamma_1\le \gamma_2$, implying $\underline X^{\gamma_1} \mid \underline X^{\gamma_2}$, contradicting the minimality of $G(K)$. Hence $\mu(K)=1$. By symmetry, $\mu(J)=1$ as well. Moreover, since $J,K\ne R$, their generators are nonconstant. This proves that the submonoid of $\Mon(R)$ consisting of principal monomial ideals is divisor-closed. Statement \ref{lem: princ div closed 1} follows from the fact that the monoid of monomials $\mathcal M(R)$ is free abelian on $X_1, \ldots,X_N$. Let $i\in [1,N]$ and take $a\in\mathbb N^N$ such that $a(i)=n$ and $a(j)=0$ for all $j\in [1,N]\setminus \{i\}$. It then follows from \ref{lem: princ div closed 1} that $\mathsf Z_{\Mon(R)}(\langle X_i^n\rangle) = \{ \langle X_i\rangle^{n}\}$, i.e., that $\langle X_i\rangle$ is a strong atom of $\Mon(R)$. This proves $\ref{lem: princ div closed 2}$.
    \end{proof}

\begin{proposition}\label{prop: princ factors}
    Let $I\in\Mon(R)$ and $\underline X^a\in \mathcal M(R)$. Then
    \[ \mathsf Z_{\Mon(R)} (\langle \underline X^a\rangle I) = \mathsf Z_{\Mon(R)}(\langle \underline X^a\rangle) \mathsf Z_{\Mon(R)}(I) = \langle X_1\rangle^{a(1)}\dots\langle X_N\rangle^{a(N)} \mathsf Z_{\Mon(R)}(I)\]
\end{proposition}
\begin{proof}
    If $a=\underline 0$, then $\langle \underline X^a\rangle=R$ and the claim is trivial. Assume then $a>\underline 0$. By Lemma \ref{lem: princ div closed}\ref{lem: princ div closed 1}, it suffices to prove the first equality. We do it for $\underline X^a = T\in\{X_1,\dots,X_N\}$, as the general case follows by iteration. It is clear that adjoining $\langle T\rangle$ to any factorization of $I$ yields a factorization of $\langle T\rangle I$, so it remains to prove the converse.

    Suppose that 
    \[\langle T\rangle I = U_1\cdots U_\ell,\]
    where $U_1,\dots,U_\ell$ are atoms of $\Mon(R)$. We claim that there exists a $j \in [1,\ell]$ such that $T\mid v$ for every $v\in G(U_j)$. Otherwise, for every $j\in[1,\ell]$ we could choose $v_j\in G(U_j)$ such that $T\nmid v_j$. Therefore,
    \[v_1\cdots v_\ell\in  U_1\cdots U_\ell=\langle T\rangle I\]
    while $T\nmid v_1\cdots v_\ell$, contradicting the fact that every monomial in $\langle T\rangle I$ is divisible by $T$. 

    Without loss of generality, assume that $T\mid v$ for every $v\in G(U_1)$. Then $U_1=\langle T\rangle U_1'$ for some $U_1'\in\Mon(R)$. Since $U_1$ is an atom, and $\langle T\rangle\ne R$, we must have $U_1'=R$, and so $U_1=\langle T\rangle$. Therefore,
    \[\langle T\rangle I = \langle T\rangle U_2\cdots U_\ell.\]
    As nonzero principal ideals are cancellative elements in $\mathcal I(R)$ (and so in $\Mon(R)$) by \cite{An-Ro97}, we obtain \[I = U_2\cdots U_\ell.\] Thus, every factorization of $\langle T\rangle I$ is obtained by adjoining $\langle T\rangle$ to a factorization of $I$, and this proves the claim.
\end{proof}

The following corollary is an immediate consequence of the above proposition.
\begin{corollary}
    For any $I\in\Mon(R)$, let $d=\gcd{(G(I))}$ and $J=\frac {1}{d} I$. Then, $J\in\Mon(R)$ and we have
    \[ \mathsf Z_{\Mon(R)}(I) = \mathsf Z_{\Mon(R)}(\langle d\rangle) \mathsf Z_{\Mon(R)}(J).  \]
\end{corollary}

Proposition \ref{prop: princ factors} also yields a translation property for sets of lengths: multiplying by a principal monomial ideal shifts all factorization lengths by the length of its unique factorization. This gives the following realization result.

\begin{proposition}\label{prop: relization of SOL}
  Let $m,n\in \mathbb N$ such that $2\le m\le n$. Then, there exists a monomial ideal $I\in \Mon(R)$ such that $\mathsf L_{\Mon(R)}(I)=[m,n]$.
\end{proposition}
\begin{proof}
It was proved in \cite[Theorem 4.7(3)]{BCK} that, for every integer $k\ge 2$,
\[ \mathsf L_{\Mon(R)}(\langle X_1,X_2\rangle^k)=[2,k].\]
Then, given \(2\le m\le n\), set $k=n-m+2$ and multiply $\langle X_1,X_2\rangle^k$ by $\langle X_1\rangle^{m-2}$. As $\langle X_1\rangle^{m-2}$ has a unique factorization of length $m-2$, we get from Proposition \ref{prop: princ factors} that
\[ \mathsf L_{\Mon(R)} \bigl(\langle X_1\rangle^{m-2}\langle X_1,X_2\rangle^{n-m+2}\bigr) = (m-2)+[2,n-m+2] = [m,n].\]
\end{proof}

\subsection{Characterizing \lq\lq small atoms\rq\rq}\label{char atom mu le 4}
In this section we provide a complete characterization of the atoms of $\Mon(R)$ with at most four minimal generators. The only atoms among the principal monomial ideals of $R$ are $\langle X_1\rangle, \dots, \langle X_N\rangle$ by Lemma \ref{lem: princ div closed}. The following results give necessary and sufficient conditions for $I=\langle\underline X^{a_1},\dots,\underline X^{a_{\mu(I)}}\rangle$ to be a (non-)atom in $\Mon(R)$, when $2\le \mu(I)\le 4$. We start with a couple of remarks establishing the theoretical and notational framework for our discussion, which will be used in the subsequent results without further comment.

\begin{remark}\label{rem: E1 E2}
    Let $I,J,K\in\Mon(R)\setminus\{R\}$ with $G(I) = \{\underline X^{a_1},\dots,\underline X^{a_{\mu(I)}}\}$, $G(J) = \{ \underline X^{\beta_1},\dots,\underline X^{\beta_{\mu(J)}} \}$, and $G(K) = \{ \underline X^{\gamma_1},\dots,\underline X^{\gamma_{\mu(K)}} \}$. By the minimality of the generating sets $G(I),G(J),G(K)$, we have that $a_i\perp a_j$ for every $i\neq j$, $\beta_k\perp\beta_\ell$ for $k\neq \ell$, and $\gamma_r\perp\gamma_s$ for $r\neq s$.

    With this notation, by Lemma \ref{lem: mon_gen}, the equality $I=JK$ holds if and only if the following two conditions are satisfied:
    \begin{enumerate}[label=\textup{\bf(E\arabic{*})}, mode=unboxed]
        \item\label{eq: I in JK} for every $i\in [1,\mu(I)]$, there exist $j_i\in [1,\mu(J)]$ and $k_i\in [1,\mu(K)]$ such that $a_i = \beta_{j_i}+\gamma_{k_i}$;
        \item\label{eq: JK in I} for every $j\in [1,\mu(J)]$ and every $k\in [1,\mu(K)]$, there is an $i\in [1,\mu(I)]$ such that $a_i \le \beta_j+\gamma_k$.
    \end{enumerate}
\end{remark}

\begin{remark}\label{rem: mon1 R}
    Let $I$ be a monomial ideal with $G(I) = \{\underline X^{a_1},\dots,\underline X^{a_{\mu(I)}}\}$. When studying conditions for $I$ to be an atom, we may assume that $I\neq R$, since atoms must be non-units. Moreover, in view of the preceding discussion, we may assume that $\gcd(G(I))=1$; otherwise, $I$ is divisible by a nontrivial principal monomial ideal. In terms of the exponent vectors, these conditions are equivalent to assuming $a_i> \underline 0$ for every $i\in [1,\mu(I)]$ and $\bigwedge_{i=1}^{\mu(I)}a_i=\underline 0.$ We denote by
    \[
        \Mon_1(R):=
        \{I\in\Mon(R):\gcd(G(I))=1\}
    \]
    the submonoid of $\Mon(R)$ consisting of monomial ideals with relatively prime minimal generators. By Lemma~\ref{lem: gcd IJ}, it is divisor-closed: indeed, if $IJ\in\Mon_1(R)$ for some $I,J\in\Mon(R)$, then
    \[
        1=\gcd(G(IJ))=\gcd(G(I))\gcd(G(J)),
    \]
    and hence
    \[
        \gcd(G(I))=\gcd(G(J))=1.
    \]
    Therefore, if $I=JK$ and $I\in\Mon_1(R)$, then, keeping the notation of Remark \ref{rem: E1 E2}, $\bigwedge_{j=1}^{\mu(J)} \beta_j = \bigwedge_{k=1}^{\mu(K)} \gamma_k = \underline 0$.
    Finally, $R$ is the unique principal ideal belonging to $\Mon_1(R)$.
    Indeed, if $\mu(I)=1$, say $G(I)=\{\underline X^{a}\}$, then $\gcd(G(I))=1$ if and only if $a=0$, equivalently, if and only if $I=R$. Consequently,
    \[
        I\in\Mon_1(R)\setminus\{R\}
        \quad\Longrightarrow\quad
        \mu(I)\ge 2.
    \]
\end{remark}

\bigskip

When $\mu(I)=2$, the next theorem shows that $I$ is an atom if and only if $I\in \Mon_1(R)$:

\begin{theorem}\label{thm: 2-gen}
    Let $I\in \Mon(R)$ with $G(I) = \{\underline X^{a_1},\underline X^{a_{2}}\}$. Then $I$ is not an atom of $\Mon(R)$ if and only if $a_1\wedge a_2>\underline 0$.
\end{theorem}
\begin{proof}
    Since $a_1,a_2\in\mathbb N^N$, we have $a_1\wedge a_2\in\mathbb N^N$, i.e., $a_1\wedge a_2\ge \underline 0$. If $a_1\wedge a_2 > \underline 0$, then, using that $a_1\perp a_2$, it is easy to see that
   \[J=\langle \underline X^{a_1\wedge a_2}\rangle \quad \text{ and }\quad K=\langle \underline X^{a_1-a_1\wedge a_2}, \underline X^{a_2-a_1\wedge a_2} \rangle \]
   are non-units of $\Mon(R)$. Thus, $I=JK$ is not an atom.
    On the other hand, assume that $I=JK$ for some $J,K\in\Mon(R)\setminus\{R\}$, and keep the notation of Remark \ref{rem: E1 E2}. By condition \ref{eq: I in JK}, there are $j_1,j_2\in [1,\mu(J)]$ and $k_1,k_2\in [1,\mu(K)]$ such that $a_1 = \beta_{j_1}+\gamma_{k_1}$ and $a_2=\beta_{j_2}+\gamma_{k_2}$. By condition \ref{eq: JK in I}, we have $a_i\le \beta_{j_1}+\gamma_{k_2}$ for either $i=1$ or $i=2$. Since $a_i=\beta_{j_i}+\gamma_{k_i}$, we obtain
    \[ \beta_{j_i}+\gamma_{k_i} \le \beta_{j_1}+\gamma_{k_2}, \]
    whence either $\gamma_{k_1}\le \gamma_{k_2}$ or $\beta_{j_2}\le \beta_{j_1}$. By minimality of $G(J)$ and $G(K)$, we obtain either $k_1=k_2$ or $j_2=j_1$. In both cases, by condition \ref{eq: I in JK} we have $a_1\wedge a_2>\underline 0$.
\end{proof}

The next theorem gives a complete characterization of (non-)atoms among the elements of $\Mon(R)$ with $3$ minimal generators. 
\begin{theorem}\label{thm: 3-generated}
Let $I\in \Mon_1(R)$ with $G(I) = \{\underline X^{a_1},\underline X^{a_{2}}, \underline X^{a_3}\}$. Then $I$ is not an atom of $\Mon(R)$ if and only if there exists a permutation $\sigma\in S_3$ such that $a_{\sigma(1)}+a_{\sigma(2)}\ge 2a_{\sigma(3)}$.
\end{theorem}
\begin{proof}
Assume that $I$ is not an atom of $\Mon(R)$. As $I\ne R$, this means that $I=JK$ for some $J,K\in\Mon(R)\setminus\{R\}$. Using the notation of Remark \ref{rem: E1 E2}, 
by \ref{eq: I in JK}, there exist $j_1,j_2,j_3\in [1,\mu(J)]$ and $k_1,k_2,k_3\in [1,\mu(K)]$ such that $a_i=\beta_{j_i}+\gamma_{k_i}$ for every $i\in [1,3]$. We proceed by analyzing the cardinality of the index sets $\{j_1,j_2,j_3\}$ and $\{k_1,k_2,k_3\}$.

\medskip
\noindent\textbf{Case 1: $\lvert\{j_1,j_2,j_3\}\rvert = 1$.} 
This is impossible, as $j_1=j_2=j_3$ implies $\underline{0} < \beta_{j_1} \le a_1 \wedge a_2 \wedge a_3$ by \ref{eq: I in JK}, a contradiction.

\medskip
\noindent\textbf{Case 2: $\lvert\{j_1,j_2,j_3\}\rvert = 2$.} 
Assume without loss of generality that $j_1 \neq j_2 = j_3$, which implies $k_2 \neq k_3$ (otherwise $a_2=a_3$ by \ref{eq: I in JK}). 

\begin{itemize}
    \item Assume $k_1=k_2$. Condition \ref{eq: JK in I} forces $a_i \le \beta_{j_1} + \gamma_{k_3}$ for some $i\in[1,3]$. 
    By \ref{eq: I in JK}, both $i=1$ and $i=3$ yield immediate contradictions ($\gamma_{k_1} \le \gamma_{k_3}$ and $\beta_{j_3} \le \beta_{j_1}$, respectively). Thus $a_2 \le \beta_{j_1} + \gamma_{k_3}$, which leads to:
    \[ a_1+a_3 = \beta_{j_1}+\gamma_{k_1}+\beta_{j_3}+\gamma_{k_3}=\beta_{j_1}+\gamma_{k_2}+\beta_{j_2}+\gamma_{k_3} \ge 2a_2 .\]
    \item The subcase $k_1 = k_3$ is perfectly analogous to the previous one.
    \item Assume that $k_1,k_2, k_3$ are pairwise distinct. By \ref{eq: JK in I}, we have $a_i\le \beta_{j_2}+\gamma_{k_1}$ for some $i\in [1,3]$. If $i=1$, we obtain $a_1=\beta_{j_1}+\gamma_{k_1}\le \beta_{j_2}+\gamma_{k_1}$ by \ref{eq: I in JK}, implying $\beta_{j_1}\le \beta_{j_2}$, against $j_1\neq j_2$. Similarly, one finds a contradiction in each of the subcases $i=2$ and $i=3$.
\end{itemize}

\medskip
\noindent\textbf{Case 3:$\lvert\{ j_1,j_2,j_3\}\rvert=3$.}
The subcase $\lvert\{k_1,k_2,k_3\}\rvert \le 2$ is ruled out by symmetric arguments to Cases 1 and 2 (exchanging the roles of $J$ and $K$). Assume then $\lvert\{k_1,k_2,k_3\}\rvert = 3$. By \ref{eq: JK in I}, for any pair $(j_r, k_s)$ with $r \neq s$, there exists $a_i \le \beta_{j_r} + \gamma_{k_s}$. If $i=r$ or $i=s$, we obtain a contradiction via \ref{eq: I in JK} (e.g., $a_1 \le \beta_{j_1}+\gamma_{k_2} \implies \gamma_{k_1} \le \gamma_{k_2}$). 
Thus, the index $i$ must be the unique element in $\{1,2,3\} \setminus \{r,s\}$. By cyclic permutations, this yields:
\begin{equation}\label{eq: mu3 system}
    a_3 \le \beta_{j_1}+\gamma_{k_2}, \quad a_1 \le \beta_{j_2}+\gamma_{k_3}, \quad a_2 \le \beta_{j_3}+\gamma_{k_1}.
\end{equation}
but also
\begin{equation}\label{eq: mu3 system 2}
    a_3 \le \beta_{j_2}+\gamma_{k_1}, \quad a_1 \le \beta_{j_3}+\gamma_{k_2}, \quad a_2 \le \beta_{j_1}+\gamma_{k_3}.
\end{equation}
Summing the relations \eqref{eq: mu3 system} gives:
\[
    a_1+a_2+a_3 \le (\beta_{j_1}+\gamma_{k_2}) + (\beta_{j_2}+\gamma_{k_3}) + (\beta_{j_3}+\gamma_{k_1}) = a_1+a_2+a_3,
\]
which forces all inequalities in \eqref{eq: mu3 system} to hold as equalities. In particular, we have $a_3 = \beta_{j_1}+\gamma_{k_2}$. Proceeding analogously with the relations in \eqref{eq: mu3 system 2}, we get $a_2 = \beta_{j_1}+\gamma_{k_3}$. Thus, combining the two above equalities with $a_1 = \beta_{j_1}+\gamma_{k_1}$ from \ref{eq: I in JK}, we obtain $\underline{0} < \beta_{j_1} \le a_1 \wedge a_2 \wedge a_3$, a contradiction.

Vice versa, let us show that $a_{\sigma(1)}+a_{\sigma(2)}\ge 2a_{\sigma(3)}$ for some $\sigma\in S_3$ is also a sufficient condition for $I=\langle \underline X^{a_1}, \underline X^{a_2}, \underline X^{a_3} \rangle$ to be a non-atom of $\Mon(R)$. Up to renaming the exponents of the generators of $I$, we can assume that $a_1+a_2\ge 2a_3$. Note that $a_1\wedge a_3>\underline 0$, otherwise $a_2\ge 2a_3 > a_3$, against $a_2\perp a_3$. Analogously, $a_2\wedge a_3>\underline 0$. We can therefore define the proper ideals:
\[
    J = \langle \underline{X}^{a_1 \wedge a_3}, \underline{X}^{a_2-a_3+a_1 \wedge a_3} \rangle, \qquad
    K = \langle \underline{X}^{a_1 - a_1 \wedge a_3}, \underline{X}^{a_3 - a_1 \wedge a_3} \rangle.
\]
Indeed, no generator exponent equals $\underline{0}$ because $a_1, a_3 > a_1 \wedge a_3$ (otherwise $a_3\ge a_1$ or $a_1\ge a_3$), and for every component $i \in [1,N]$:
\[
    a_2(i) + (a_1 \wedge a_3)(i) = a_2(i)+a_1(i)\ge 2a_3(i) \ge a_3(i),
\]
or
\[
    a_2(i) + (a_1 \wedge a_3)(i) = a_2(i)+a_3(i) \ge a_3(i),
\]
with at least one strict inequality since $a_2 \wedge a_3 > \underline{0}$. Finally, we verify that $JK = I$. Clearly $\underline{X}^{a_1}, \underline{X}^{a_2}, \underline{X}^{a_3} \in JK$. The remaining product generator satisfies:
\[
    (a_2 + a_1 \wedge a_3 - a_3) + (a_1 - a_1 \wedge a_3) = a_2 + a_1 - a_3 \ge 2a_3 - a_3 = a_3,
\]
hence $\underline{X}^{a_2+a_1-a_3} \in \langle \underline{X}^{a_3} \rangle \subseteq I$, completing the proof.
\end{proof}

We conclude this section by characterizing the non-atoms of $\Mon(R)$ with $4$ minimal generators. The proof proceeds by a case analysis and relies on arguments analogous to those detailed in the proof of the previous theorem. To avoid unnecessary repetition, we omit routine symmetric arguments and give details only when new features arise.

\begin{theorem}\label{thm: 4-generated}
Let $I\in \Mon_1(R)$ with $G(I) = \{\underline X^{a_1},\underline X^{a_{2}}, \underline X^{a_3}, \underline X^{a_4}\}$. Then $I$ is not an atom of $\Mon(R)$ if and only if it satisfies one of the following conditions, up to a relabeling of the exponents of its generators:
\begin{enumerate}[label=\textup{(c\arabic{*})}, mode=unboxed]
 \item \label{mu4 c1} $a_1+a_3=a_2+a_4$; 
 \item \label{mu4 c2} $a_1+a_3\ge 2a_4$ and $a_2+a_4\ge 2a_1$; 
 \item \label{mu4 c3} $a_1\wedge a_2\wedge a_3>\underline 0$ and one of the following holds:
        \begin{enumerate}[label=\textup{(\roman{*})}, mode=unboxed]
            \item \label{mu4 c31} $a_2+a_4\ge a_1+a_3$ and $a_3+a_4\ge 2a_1$;
            \item \label{mu4 c32} $a_2+a_4\ge 2a_1$ and $a_3+a_4\ge a_1+a_2$;
            \item \label{mu4 c33} $a_2+a_4\ge 2a_1$ and $a_3+a_4\ge 2a_1$.
        \end{enumerate}
    \end{enumerate}
\end{theorem}

\begin{proof}
Suppose that $I$ is not an atom of $\Mon(R)$, i.e., $I=JK$ for some $J,K\in\Mon(R)\setminus\{R\}$. We use the notation introduced in Remark \ref{rem: E1 E2}: for every $i\in[1,4]$ we choose indices $j_i,k_i$ such that
\[
a_i=\beta_{j_i}+\gamma_{k_i}.
\]
Recall that whenever $j_r\neq j_s$, condition \ref{eq: JK in I} applied to the crossed sum $\beta_{j_r}+\gamma_{k_s}$ cannot hold for $a_s$, as this would contradict the minimality of $G(J)$. Similarly, whenever $k_r\neq k_s$, this sum cannot dominate $a_r$ due to the minimality of $G(K)$. Consequently, if $j_r\neq j_s$ and $k_r\neq k_s$, the crossed sum $\beta_{j_r}+\gamma_{k_s}$ must dominate one of the remaining generators. We will use this observation repeatedly without further mention.

Up to relabeling the $a_i$'s and the generators of $J$ and $K$, and possibly interchanging $J$ and $K$, we may assume $|\{j_1,j_2,j_3,j_4\}|\le|\{k_1,k_2,k_3,k_4\}|$. We then distinguish the following cases, depending on the cardinality and multiplicity pattern of the index set $\{j_i\}_{i=1}^4$.
\begin{enumerate}[label=\textup{(\arabic{*})}, mode=unboxed]
        \item\label{mu4 4000} $j_1 = j_2 = j_3 = j_4$;
        \item\label{mu4 3100}  $j_1=j_2=j_3\neq j_4$;
        
        \item\label{mu4 2200} $j_1=j_2\ne j_3=j_4$;
        
        \item\label{mu4 2110} $\lvert\{j_1,j_2,j_3,j_4\}\rvert = 3,\ j_1=j_2$;
        
        \item\label{mu4 1111} $\lvert\{j_1,j_2,j_3,j_4\}\rvert=4$.
\end{enumerate}
Recall that, by the minimality of $G(I)$, if $j_r=j_{s}$ for distinct $r,s\in[1,4]$, then $k_r\ne k_{s}$.

\smallskip
\textbf{Case \ref{mu4 4000}.} This case cannot occur, as it implies $a_1\wedge a_2\wedge a_3\wedge a_4 \ge \beta_{j_1} > \underline 0$, against our assumptions.

\smallskip
\textbf{Case \ref{mu4 3100}.} Suppose $j_1=j_2=j_3\neq j_4$. If $k_4\notin\{k_1,k_2,k_3\}$, then the crossed sum $\beta_{j_1}+\gamma_{k_4}$ cannot dominate any $a_i$, contradicting \ref{eq: JK in I}. Hence, after relabeling, we may assume $k_4=k_1$. Then $a_1\wedge a_2\wedge a_3\ge \beta_{j_1}>\underline 0.$ By \ref{eq: JK in I} there are $i\in \{1,3\}$ and $\ell\in\{1,2\}$ such that
\[\beta_{j_4}+\gamma_{k_2}\ge a_i \quad\text{ and }\quad \beta_{j_4}+\gamma_{k_3}\ge a_\ell.\] The choices $(i, \ell) = (3, 1), (1, 2), (1, 1)$ yield, respectively, \ref{mu4 c3}\ref{mu4 c31}, \ref{mu4 c3}\ref{mu4 c32}, and \ref{mu4 c3}\ref{mu4 c33}. The choice $(i,\ell)=(3,2)$ yields $a_4\ge a_1$, a contradiction.

\smallskip
\textbf{Case \ref{mu4 2200}.} Suppose that $j_1=j_2\neq j_3=j_4$. If the four $k_i$'s are distinct, the forced crossed inequalities imply simultaneously
\[
a_2+a_3\ge a_1+a_4,\qquad
a_2+a_4\ge a_1+a_3,
\]
and hence $a_2\ge a_1$, a contradiction. If, after relabeling, $k_1=k_4$ and $k_2\neq k_3$, then the crossed sums give
\[
a_1+a_3\ge 2a_4,\qquad a_2+a_4\ge 2a_1,
\]
which is condition \ref{mu4 c2}. Finally, if $k_1=k_4$ and $k_2=k_3$, then immediately
\[
a_1+a_3=a_2+a_4,
\]
which is condition \ref{mu4 c1}.

\smallskip
\textbf{Case \ref{mu4 2110}.} Assume that $j_1=j_2$ and $j_3,j_4$ are distinct from each other and from $j_1$. If the four $k_i$'s are distinct, a repeated application of the crossed-sum observation forces a chain of inequalities in which every alternative except one leads to a forbidden comparison between two of the $a_i$'s. In the remaining alternative, the inequalities are forced into equalities, yielding
\[
a_1\wedge a_2\wedge a_3\wedge a_4\ge \beta_{j_1}>\underline 0,
\]
a contradiction. If, on the other hand, the set $\{k_i\}_{i=1}^4$ has cardinality $3$, say $k_3=k_4\notin\{k_1,k_2\}$, then the crossed sum $\beta_{j_1}+\gamma_{k_3}$ cannot dominate any of the $a_i$, again a contradiction. In the remaining subcase, after relabeling, $k_2=k_3$ and $k_4\notin\{k_1,k_2\}$. In this case the crossed sums first give
\[
a_2+a_4\ge a_1+a_3.
\]
The remaining relevant crossed sum $\beta_{j_3} +\gamma_{k_1}$ cannot dominate $a_2$, otherwise $a_4\ge a_2$, thus it must dominate $a_4$ and the previous inequality becomes an equality:
\[
a_2+a_4=a_1+a_3.
\]
Thus condition \ref{mu4 c1} holds.

\textbf{Case \ref{mu4 1111}.} In this case, all the $j_i$'s are pairwise distinct, and hence so are all the $k_i$'s. For $i\neq \ell$, choose
\[
\nu(i,\ell)\in[1,4]\setminus\{i,\ell\}
\quad \text{ such that }\quad
\beta_{j_i}+\gamma_{k_\ell}\ge a_{\nu(i,\ell)}.
\]

We first claim that $\nu(i,\ell)=\nu(\ell,i)$
for all $i\neq \ell$. Indeed, if $\nu(i,\ell)=m$ and $\nu(\ell,i)=n$ are distinct, then $\{i,\ell,m,n\}=[1,4]$ and
\[
a_i+a_\ell
\ge a_m+a_n
\ge a_{\nu(m,n)}+a_{\nu(n,m)}.
\]
The possible choices for $\nu(m,n)$ and $\nu(n,m)$ either immediately contradict incomparability of the $a_i$'s, or force equality throughout. In the latter case one obtains identities of the form
\[
\beta_{j_i}+\gamma_{k_\ell}=a_m,\qquad
\beta_{j_i}+\gamma_{k_n}=a_\ell,\qquad
\beta_{j_i}+\gamma_{k_m}=a_n,
\]
and hence
\[
a_1\wedge a_2\wedge a_3\wedge a_4\ge \beta_{j_i}>\underline 0,
\]
again impossible. This proves the claim that $\nu(i,\ell)=\nu(\ell,i)$ for all $i\ne\ell$. It remains to consider the possible assignments of these values. Up to a relabeling of the indices, we may fix $\nu(1,2)$. A finite case analysis, repeatedly applying the crossed-sum observation, shows that every assignment either yields a forbidden comparison $a_r\ge a_s\qquad (r\ne s)$, or forces equality in a sum of crossed inequalities and consequently contradicts the minimality of $G(J)$ or $G(K)$. Therefore this case cannot occur. 

We have proved that if $I$ is not an atom, then one of \ref{mu4 c1}, \ref{mu4 c2}, \ref{mu4 c3} holds.

Conversely, suppose first that \ref{mu4 c1} holds, say $a_1+a_3=a_2+a_4.$ Set $d=a_1\wedge a_2$. Then
\[
I=
\langle \underline X^d, \underline X^{a_4-a_1+d}\rangle
\langle \underline X^{a_1-d},\underline X^{a_2-d}\rangle .
\]
Indeed, the four products give $a_1,a_2,a_4$, and $a_4-a_1+d+a_2-d=a_4-a_1+a_2=a_3.$ Moreover, the exponents involved are all $>\underline 0$: for instance, $d=\underline 0$ would imply $a_3\ge a_2$, and $a_4-a_1+d=\underline 0$ would imply $a_1\ge a_4$, both contradicting the minimality of $G(I)$.

Suppose next that \ref{mu4 c2} holds, say $a_1+a_3\ge 2a_4$ and $a_2+a_4\ge 2a_1.$ Set $d=a_1\wedge a_4$. Then
\[
I=
\langle \underline X^{a_1-d},\underline X^{a_4-d}\rangle
\langle \underline X^d,\underline X^{a_2-a_1+d},\underline X^{a_3-a_4+d}\rangle .
\]
The inequalities ensure that the extra products are absorbed:
\[
a_4+a_2-a_1\ge a_1,\qquad a_1+a_3-a_4\ge a_4.
\]
The same component-wise check as above shows that all displayed vector exponents are $>\underline 0$.

Finally, assume \ref{mu4 c3}. Put $d=a_1\wedge a_2\wedge a_3>\underline 0.$ Then
\[
I=
\langle \underline X^d,\underline X^{a_4-a_1+d}\rangle
\langle \underline X^{a_1-d},\underline X^{a_2-d},\underline X^{a_3-d}\rangle .
\]
The products with $\underline X^d$ give $a_1,a_2,a_3$, while the product $(a_4-a_1+d)+(a_1-d)$ gives $a_4$. The remaining two products have exponents
\[
a_4-a_1+a_2,\qquad a_4-a_1+a_3,
\]
and these dominate one of $a_1,a_2,a_3$ precisely by the alternatives listed in \ref{mu4 c3}. Thus they are absorbed by $I$. Again, the hypotheses and the incomparability of the $a_i$'s ensure that all exponents are $>\underline 0$. Hence in each case $I$ admits a nontrivial factorization, and so it is not an atom.
\end{proof}

\section{Density of atoms among \lq\lq small\rq\rq\ ideals}

Density questions for atoms have received comparatively little attention in factorization theory. Related results include zero-density phenomena for atoms in analytic Krull monoids \cite[Chapter~9]{Ge-HK} and in numerical semigroup algebras \cite{AEKOT22}, as well as recent density results for atoms in power monoids \cite{BiGe25, Agg-Got-Lu-2025}. Here we consider a different notion of density, tailored to monomial ideals. Using the results of Subsection \ref{char atom mu le 4}, we study atoms among monomial ideals $I$ of $R=K[X_1,\dots,X_N]$ with $\mu(I)\in [1,4]$ and $\gcd{G(I)}=1$. In particular, we show that for fixed $\mu\in[2,4]$, the corresponding density goes to $1$ as $N$ goes to infinity.

Let us first define what we mean by density. Given the number of variables of the underlying polynomial ring, $N$, and fixing the number of minimal generators of the ideals to be considered, $\mu$, we define, for any $M\in\mathbb N^+$, the set
\[ \mathcal T_M(N,\mu) := \{I\in\Mon_1(R)\colon \mu(I)=\mu, \forall \underline X^a\in G(I)\ a\le (M,\dots,M)\}. \]
If we further denote by $\mathcal A_M(N,\mu) $ the set $\mathcal T_M(N,\mu) \cap \mathscr A(\Mon(R))$, we can define the \emph{density} of atoms among $N$-variate monomial ideals with $\mu$ generators as the following limit, provided it exists
\[ d(N,\mu) = \lim_{M\to\infty} \frac{\lvert \mathcal A_M(N,\mu)\rvert}{\lvert \mathcal T_M(N,\mu)\rvert}.\]
For convenience, let us also introduce the set $\mathcal F_M(N,\mu):=\mathcal T_M(N,\mu)\setminus {\mathscr A(\Mon(R))}=\mathcal T_M(N,\mu)\setminus\mathcal A_M(N,\mu)$, so that 
\[ \frac{\lvert \mathcal A_M(N,\mu)\rvert}{\lvert \mathcal T_M(N,\mu)\rvert} = 1-\frac{\lvert \mathcal F_M(N,\mu)\rvert}{\lvert \mathcal T_M(N,\mu)\rvert}. \]

\begin{remark}
For $\mu\in\{1,2\}$, the density $d(N,\mu)$ can be determined explicitly from the definition. Indeed, $R$ is the unique element of $\Mon_1(R)$ with one minimal generator, and it is not an atom. Hence $d(N,1)=0$ for every $N\ge2$. On the other hand, Theorem~\ref{thm: 2-gen} shows that every ideal $I\in\Mon_1(R)$ with $\mu(I)=2$ is an atom of $\Mon(R)$, and therefore $d(N,2)=1$ for every $N\ge2$. Thus, $\mu=3$ is the first nontrivial case.
\end{remark}

\begin{remark}
One could define the sets $\mathcal T_M(N,\mu)$ (and consequently, $\mathcal A_M(N,\mu)$ and $\mathcal F_M(N,\mu)$) by considering ideals of $\Mon(R)$. However, since all ideals in $\Mon(R)\setminus\Mon_1(R)$ with $\mu\ge 2$ are non-atoms, it is easily seen that this choice would lead to $d(N,\mu)=0$ for any $N$ and any $\mu\ge 2$.
\end{remark}

\begin{theorem}\label{thm: d3}
    The limit $d(N,3)$ exists for every $N\ge 2$ , and
    \[ \lim_{N\to\infty} d(N,3) = 1.\]
    In particular, $d(2,3)=\frac34$.
\end{theorem}
\begin{proof}
Since we are concerned only with ideals with $\mu=3$ minimal generators, we omit the dependence on $\mu$ throughout the proof; e.g., we denote $d(N,3)$ simply by $d(N)$. We first compute the value of $d(2)$. Let
\[ I\in \mathcal T_M(2)\text{ with } G(I)=\{\underline X^{a_1},\underline X^{a_2},\underline X^{a_3}\},\]
where $a_i=(b_i,c_i)\in[0,M]^2\setminus\{(0,0)\}$. Since $I\in\Mon_1(R)$ and $\mu(I)=3$, the exponents are pairwise incomparable and satisfy
\[ b_1b_2b_3=0,\qquad c_1c_2c_3=0.\]
Up to a relabeling of the three generators, every such triple is uniquely of the form
\[ a_1=(x,0),\qquad a_2=(u,v),\qquad a_3=(0,y),\]
with
\[ 1\le u<x\le M,\qquad 1\le v<y\le M.\]
Therefore, we have
\[ |\mathcal T_M(2)|=\binom{M}{2}^2=\frac14M^4 + O(M^3).\]

By Theorem \ref{thm: 3-generated}, an ideal of the above form is not an atom if and only if
\[ a_1+a_3\ge 2a_2, \]
(it is readily checked that neither $a_1+a_2\ge 2a_3$ nor $a_2+a_3\ge 2a_1$ can occur), which is equivalent to
\[ x\ge 2u,\qquad y\ge 2v.\]
The number of pairs $(u,x)\in [1,M]^2$ satisfying $x\ge 2u$ is
\[\sum_{u=1}^{\lfloor M/2\rfloor}(M-2u+1) = \frac14M^2+O(M).\]
An analogous formula holds for the pairs $(v,y)$ satisfying $y\ge 2v$, whence 
\[ |\mathcal F_M(2)| = \left(\frac14M^2+O(M)\right)^2 = \frac1{16}M^4+O(M^3).\]

We therefore obtain
\[ \lim_{M\to\infty} \frac{|\mathcal F_M(2)|}{|\mathcal T_M(2)|} = \frac{1/16}{1/4} = \frac14, \]
and consequently
\[ d(2)=1-\frac14=\frac34.\]

We now prove the existence of the limit $d(N)$ for every $N\ge 2$ and show that $d(N)\to 1$ as $N\to\infty$. It will be convenient to count ordered triples of exponent vectors; this does not change the relevant ratios, since in this way every ideal is counted exactly $3!$ times. With a slight abuse of notation, we continue to denote the corresponding counts by $|\mathcal T_M(N)|$ and $|\mathcal F_M(N)|$. Fix the number of variables $N$. Given an ordered triple
\[(a_1,a_2,a_3)\in [0,M]^N\times [0,M]^N\times [0,M]^N\]
satisfying
\begin{equation}\label{densita 3 dom}
    a_1(i)a_2(i)a_3(i)=0\qquad\text{for}\ i\in[1,N],
\end{equation}
we define its \emph{zero scheme} (or \emph{scheme}, for short) to be
\[ Z=(Z_1,\dots,Z_N),\ \text{where}\ Z_i:=\{k\in [1,3]:a_k(i)=0\}.\]
Each $Z_i$ is a nonempty subset of $[1,3]$, so there are $(2^3-1)^N=7^N$ schemes in total. For a fixed scheme $Z$, the number of positive entries in any ordered triple satisfying \eqref{densita 3 dom} with scheme $Z$ is
\[ p(Z) := \sum_{i=1}^N(3-|Z_i|). \]
Thus, by the independence of the coordinates, the number of such triples is exactly $M^{p(Z)}$. Note  that $p(Z)\le 2N$, with equality holding if and only if $|Z_i|=1$ for every $i\in [1,N]$; we call such schemes \emph{dominant}. A non-dominant scheme has $|Z_i|\ge 2$ for at least one $i$, so that $p(Z)\le 2N-1$. Since there are at most $7^N$ schemes in total, the number of ordered triples satisfying \eqref{densita 3 dom} with non-dominant schemes is at most $7^N M^{2N-1}$.

For each fixed dominant scheme $Z$, the incomparability conditions (other than \eqref{densita 3 dom}) and the non-atomicity conditions of Theorem \ref{thm: 3-generated} are given by finitely many homogeneous linear (in)equalities among the nonzero entries $a_k(i)$, combined disjunctively where existential quantifiers occur. Thus, the cardinalities of $\mathcal T_M(N)$ and $\mathcal F_M(N)$ restricted to the scheme $Z$ correspond to the number of lattice points (i.e., of points with integer coordinates) in some Boolean combinations (unions, intersections and complements) of finitely many hyperplanes and half-spaces in $[0,M]^{2N}\subseteq \mathbb R^{2N}$ (here $[0,M]$ is an interval of $\mathbb R$). By decomposing these regions into finitely many rational polytopes and the standard lattice-point estimate by Ehrhart for such polytopes (see, for instance, \cite[Theorem 3.23, Exercise 3.34]{Be-Ro}), both cardinalities have asymptotic expansions of the form
\[ c_Z M^{2N}+O(M^{2N-1}),\]
where the nonnegative constant $c_Z$ does not depend on $M$. Summing over all dominant schemes, and recalling that non-dominant schemes contribute at most $7^NM^{2N-1}$ to the total count (as noted above), we conclude that both $|\mathcal T_M(N)|$ and $|\mathcal F_M(N)|$ have asymptotic expansions of the form
$$|\mathcal T_M(N)|=\tau(N)\,M^{2N}+O(M^{2N-1})\quad \text{ and }\quad |\mathcal F_M(N)|=\phi(N)\,M^{2N}+O(M^{2N-1}),$$
where $\tau(N)=\sum_{Z\text{ dominant}}c_Z^T$ and $\phi(N)=\sum_{Z\text{ dominant}}c_Z^F$. Here, $c_Z^T$ and $c_Z^F$ denote the leading coefficients obtained in the previous step for the cardinalities of $\mathcal T_M(N)$ and $\mathcal F_M(N)$ restricted to the dominant scheme $Z$, respectively.
 
Note that a dominant scheme $Z$ can also be seen as a function
\[ f=f_Z:[1,N]\to[1,3],\]
where $f(i)=k$ if and only if $a_k(i)=0$ and $a_j(i)>0$ for $j\in [1,3]\setminus\{k\}$. If $f$ is surjective, then every triple $(a_1,a_2,a_3)\in([0,M]^N)^3$ with scheme $f$ automatically satisfies $a_1,a_2,a_3\neq 0$ and incomparability: for every pair $j\ne k$, there is a coordinate $i\in [1,N]$ in which $a_j(i)=0<a_k(i)$ and a coordinate $\ell$ in which $a_k(\ell)=0<a_j(\ell)$. Thus, every triple with a dominant surjective scheme actually corresponds to a monomial ideal with $\mu=3$. This shows, in particular, that the leading coefficient of $|\mathcal T_M(N)|$ is positive for every $N\ge 3$ (for which there exists at least a surjective $f$); the positivity of $\tau(2)$ follows from the explicit computation above. Therefore, the limit $d(N)=1-\phi(N)/\tau(N)$ exists for every $N\ge 2$.

It remains to prove that $d(N)\to 1$. We will estimate $\tau(N)$ from below and $\phi(N)$ from above. As triples $(a_1,a_2,a_3)$ with dominant surjective schemes correspond to monomial ideals with $3$ minimal generators, the coefficient $\tau(N)$ is at least the number of surjective functions from $[1,N]$ to $[1,3]$. By the inclusion-exclusion principle, we obtain
\begin{equation}\label{tau3 lower bound}
    \tau(N)\ge 3^N-3\cdot 2^N+3.
\end{equation}

To find an upper bound for $\phi(N)$, let us fix one of the conditions given in Theorem \ref{thm: 3-generated} guaranteeing non-atomicity, say
\begin{equation}\label{non 3atom count}
    a_1+a_2\ge 2a_3.
\end{equation}
We bound from above the number of triples $(a_1,a_2,a_3)\in([0,M]^N)^3$ satisfying this condition by ignoring the incomparability requirement and summing over all dominant schemes $f\colon [1,N]\to [1,3]$. When $f(i)=1$, in the coordinate $i$ condition \eqref{non 3atom count} reads
\[ a_2(i)\ge 2a_3(i).\]
 We have already seen that the number of pairs $(x,y)\in[1,M]^2$ such that $x\ge 2y$ is $\frac14M^2+O(M)$. The same computation holds when $f(i)=2$. If $f(i)=3$, condition \eqref{non 3atom count} in the coordinate $i$ is automatic. Summing over all dominant schemes $f$, the total number of triples $(a_1,a_2,a_3)$ satisfying \eqref{densita 3 dom} and \eqref{non 3atom count} is at most
 \begin{align*}
     & \sum_{f\colon [1,N]\to[1,3]} \biggl(\frac14 M^2+O(M)\biggr) ^{\lvert f^{-1}(1)\rvert} \biggl(\frac14 M^2+O(M)\biggr)^{\lvert f^{-1}(2)\rvert} (M^2)^{\lvert f^{-1}(3)\rvert} =\\
     &= \sum_{n_1+n_2+n_3=N} \binom{N}{n_1,n_2,n_3} \biggl(\frac14 M^2+O(M)\biggr)^{n_1}\biggl(\frac14 M^2+O(M)\biggr)^{n_2} (M^2)^{n_3} =\\
     &=\left[\left(\frac14 M^2+O(M)\right)+\left(\frac14 M^2+O(M)\right)+ M^2\right]^N=\\
     &= \biggl( \frac14+\frac14+1\biggr)^N M^{2N} + O(M^{2N-1}) =\\ 
     &=\biggl(\frac32\biggr)^N M^{2N} + O(M^{2N-1}). \\
 \end{align*}
Since there are three conditions of the form \eqref{non 3atom count} to consider, we conclude that
\begin{equation}\label{phi3 upper bound}
    \phi(N)\,M^{2N}+O(M^{2N-1})\le 3\left[ \biggl(\frac32\biggr)^N M^{2N} + O(M^{2N-1})\right] \quad\text{ so } \quad \phi(N)\le 3\biggl(\frac32\biggr)^N.
\end{equation}

Putting estimates \eqref{tau3 lower bound} and \eqref{phi3 upper bound} together, we obtain
\[ 1-d(N) = \lim_{M\to\infty}\frac{\lvert \mathcal F_M(N)\rvert}{\lvert \mathcal T_M(N)\rvert} = \frac{\phi(N)}{\tau(N)} \le \frac{3(3/2)^N}{3^N-3\cdot 2^N+3}.\]
As the right-hand side tends to $0$ for $N\to\infty$, we conclude that
\[ \lim_{N\to\infty}d(N)=1.\]
\end{proof}

The next result focuses on the density $d(N,\mu)$ when $\mu=4$. The proof technique is similar to that of Theorem \ref{thm: d3}; for this reason, we will avoid unnecessary repetitions.
\begin{theorem}
    The limit $d(N,4)$ exists for every $N\ge 2$, and
    \[ \lim_{N\to\infty} d(N,4) = 1.\]   
\end{theorem}
\begin{proof}
    We omit the dependence on $\mu=4$ throughout the proof. First, we show that $d(N)$ exists for every $N\ge 2$. As in the proof of Theorem \ref{thm: d3}, we count ordered quadruples of exponent vectors rather than ideals and we continue to denote by $|\mathcal T_M(N)|$ and $|\mathcal F_M(N)|$ the corresponding counts. As every monomial ideal with $\mu=4$ corresponds to exactly $4!$ ordered quadruples, this only rescales numerator and denominator of $\lvert \mathcal A_M(N)\rvert/\lvert\mathcal T_M(N)\rvert$ by $4!$, leaving the ratio (and hence $d(N)$) unchanged. Fix $N\ge 2$. For an ordered quadruple
    \[ (a_1,a_2,a_3,a_4)\in ([0,M]^N)^4,\]
    satisfying
    \begin{equation}\label{densita 4 dom}
        a_1\wedge a_2\wedge a_3\wedge a_4 = \underline 0,
    \end{equation}
    define its (\emph{zero}) \emph{scheme} as the $N$-tuple
    \[ Z=(Z_1,\dots,Z_N), \qquad Z_i:=\{k\in[1,4]:a_k(i)=0\}. \]
    Each $Z_i$ is a nonempty subset of $[1,4]$ by \eqref{densita 4 dom}, so there are $15^N$ schemes. For a fixed scheme $Z$, the number of positive entries in any ordered quadruple satisfying \eqref{densita 4 dom} is
    \[ p(Z):=\sum_{i=1}^N(4-|Z_i|). \]
    Thus, the number of quadruples with scheme $Z$ is exactly $M^{p(Z)}$. We have $p(Z)\le 3N$, with equality holding if and only if $|Z_i|=1$ for every $i\in [1,N]$. In this case, we call a scheme \emph{dominant}. The number of quadruples with non-dominant schemes is at most $15^N M^{3N-1}$, whence they do not affect the coefficient of $M^{3N}$ in the cardinality of either $\mathcal T_M(N)$ or $\mathcal F_M(N)$. For each fixed dominant scheme $Z$, the incomparability condition and the conditions \ref{mu4 c1}-\ref{mu4 c3} from Theorem \ref{thm: 4-generated} are expressed by finitely many homogeneous linear (in)equalities. Again, by \cite[Theorem 3.23, Exercise 3.34]{Be-Ro}, it follows that $\mathcal T_M(N)$ and $\mathcal F_M(N)$ have asymptotic expansions
    \[ \lvert\mathcal T_M(N)\rvert=\tau(N)M^{3N}+O(M^{3N-1}),\qquad \lvert\mathcal F_M(N)\rvert=\phi(N)M^{3N}+O(M^{3N-1}), \]
    where the coefficients $\tau(N),\phi(N)$ are nonnegative and independent of $M$.
    Note that dominant schemes correspond bijectively to functions
    \[ f\colon [1,N]\to [1,4], \]
    where $f(i)=k$ if and only if $a_k(i)=0$ and $a_j(i)>0$ for any $j\in [1,4]\setminus\{k\}$. Every ordered quadruple with a dominant surjective scheme satisfies $a_k\neq \underline 0$ for every $k\in [1,4]$ and $a_k\perp a_j$ for $k\neq j$, hence it determines a monomial ideal with four minimal generators. Consequently, $\tau(N)>0$ for every $N\ge 4$ (for which a surjective $f$ exists). For $N=2$, note that any $6$-uple $(b_1,b_2,c_2,b_3,c_3,c_4)\in [1,M]^6$ such that
    \begin{equation}\label{eq: mu4 n2 esiste}
        b_3<b_2<b_1 \quad\text{and}\quad c_2<c_3<c_4
    \end{equation}
    corresponds to an ideal in $\mathcal T_M(2)$, namely
    \[ I=\langle X_1^{b_1},X_1^{b_2}X_2^{c_2},X_1^{b_3}X_2^{c_3},X_2^{c_4}\rangle,\]
    and this correspondence is clearly injective. The number of elements of $[1,M]^6$ satisfying \eqref{eq: mu4 n2 esiste} is
    \[ \binom{M}{3}^2 = \frac{M^2(M-1)^2(M-2)^2}{36} = \frac{M^6}{36}+O(M^5),\]
    whence $\tau(2)\ge 1/36>0$. For $N=3$, note that any $9$-uple $(b_1,c_1,b_2,d_2,c_3,d_3,b_4,c_4,d_4)\in [1,M]^9$ such that
    \begin{equation}\label{eq: mu4 n3 esiste}
        c_1>c_4,\ b_2>b_4 \quad\text{and}\quad d_3>d_4
    \end{equation}
    corresponds to an ideal in $\mathcal T_M(3)$, namely
    \[ I=\langle X_1^{b_1}X_2^{c_1},X_1^{b_2}X_3^{d_2},X_2^{c_3}X_3^{d_3},X_1^{b_4}X_2^{c_4}X_3^{d_4}\rangle,\]
    and this correspondence is clearly injective. The number of elements of $[1,M]^9$ satisfying \eqref{eq: mu4 n3 esiste} is
    \[ M^3\binom{M}{2}^3 = \frac{M^6(M-1)^3}{8} = \frac{M^9}{8}+O(M^8), \]
    whence $\tau(3)\ge 1/8>0$. As $\tau(N)>0$ for every $N\ge 2$, we conclude that $d(N)$ exists for every $N\ge 2$ and equals \[ d(N)=1-\frac{\phi(N)}{\tau(N)}. \]
    
    It remains to show that $\phi(N)/\tau(N)\to 0$. As ordered quadruples with dominant surjective schemes correspond to monomial ideals, the inclusion-exclusion principle yields
    \[ \tau(N)\ge 4^N-4\cdot3^N+6\cdot2^N-4. \]
    We now estimate the coefficient $\phi(N)$ from above, by considering the conditions for non-atomicity of Theorem \ref{thm: 4-generated}. Note that condition \ref{mu4 c1} cannot contribute to the $M^{3N}$-term in the asymptotic expansion of $\lvert \mathcal F_M(N)\rvert$, since in every coordinate it imposes one nontrivial linear equality among the three nonzero entries. Consider any fixed relabeling of any one of the conditions \ref{mu4 c2}, \ref{mu4 c3}-\ref{mu4 c31}, \ref{mu4 c3}-\ref{mu4 c32} or \ref{mu4 c3}-\ref{mu4 c33}, and denote it by $(\ast)$. We ignore the incomparability assumption and, in condition \ref{mu4 c3}, we also ignore the requirement $a_1\wedge a_2\wedge a_3> \underline 0$. For any given coordinate $i\in [1,N]$, set $[1,4]\setminus\{k\} = \{r,s,t\}$ with $r<s<t$ and let $F_k$ denote the number of triples $(a_r(i),a_s(i),a_t(i))\in [1,M]^3$ such that condition $(\ast)$ holds in the coordinate $i$ when $f(i)=k$. Note that $F_k$ is well-defined, as it does not depend on $i$ but only on $f(i)$. No matter what condition $(\ast)$ is, it is easily checked that there is a $k\in[1,4]$ such that
    \[ F_k \le \frac14 M^3+O(M^2).\]
    For instance, for the condition
    \[ a_1+a_3\ge 2a_4,\qquad a_2+a_4\ge 2a_1, \]
    when $f(i)=3$ the first inequality reads $a_1(i)\ge 2a_4(i)$ in the coordinate $i$, whence $F_3\le \frac14 M^3+O(M^2)$ (even ignoring the second inequality). For $j\in [1,4]\setminus\{k\}$, we estimate trivially $F_j\le M^3$ to obtain
    \[ F_1+F_2+F_3+F_4 \le \biggl(1+1+1+\frac14\biggr)M^3+O(M^2) = \frac{13}{4} M^3+O(M^2).\]
    Summing over all dominant schemes, we conclude that the number of quadruples $(a_1,a_2,a_3,a_4)$ satisfying \eqref{densita 4 dom} and $(\ast)$ is at most
    \begin{align*}
        &\sum_{f\colon [1,N]\to[1,4]} \prod_{k=1}^4 F_k^{\lvert f^{-1}(k)\rvert} = \\
        = &\sum_{n_1+n_2+n_3+n_4=N} \binom{N}{n_1,n_2,n_3,n_4} \prod_{k=1}^4 F_k^{n_k} = \\
        = &(F_1+F_2+F_3+F_4)^N \\
        \le &\biggl(\frac{13}{4}\biggr)^N M^{3N}+ O(M^{3N-1}).
    \end{align*}
     Summing over all the possible choices of $(\ast)$, i.e., over all the relabelings of \ref{mu4 c2}, \ref{mu4 c3}-\ref{mu4 c31}, \ref{mu4 c3}-\ref{mu4 c32} and \ref{mu4 c3}-\ref{mu4 c33} (ignoring incomparability and $a_1\wedge a_2\wedge a_3> \underline 0$), we obtain
     \[ \phi(N)\le C\biggl(\frac{13}{4}\biggr)^N,\]
     where $C$ is the total number of possible choices for $(\ast)$, a positive integer independent of $M$ and $N$. Combining this with the lower bound for $\tau(N)$, we get
     \[ 1-d(N) = \lim_{M\to\infty}\frac{|\mathcal F_M(N)|}{|\mathcal T_M(N)|} = \frac{\phi(N)}{\tau(N)} \le \frac{C(13/4)^N}{4^N-4\cdot3^N+6\cdot2^N-4}. \]
     As the right-hand side tends to $0$ for $N\to\infty$, we finally obtain
     \[ \lim_{N\to\infty}d(N)=1. \]
\end{proof}

\section{A further density result for bivariate monomial ideals}
We conclude the paper with a density result that is, in a sense, \lq\lq orthogonal\rq\rq\ to those in the previous section. Let $\mathcal E(R)$ be the smallest divisor-closed submonoid of $\Mon(R)$ containing all the equigenerated ideals in $\Mon_1(R)$. As $\Mon_1(R)$ is divisor-closed in $\Mon(R)$, it is clear that $\mathcal E(R)\subseteq \Mon_1(R)$.

\begin{remark}\label{rmk: mon equigen ez}
    For $m\in\mathbb N$, denote by $G_m(I)$ the set of minimal generators of $I$ having degree $m$. Let
    \[\mathcal N(R) := \{I\in\Mon(R)\colon  \gcd(G_{\mdeg(I)}(I)) = 1\}. \]
    Since $\gcd(G(I))\mid \gcd(G_{\mdeg(I)}(I)),$ we have $\mathcal N(R)\subseteq \Mon_1(R)$. For $J,K \in \Mon(R)$, set
\[
K_1 := \langle G_{\mdeg(J)}(J)\rangle, \qquad K_2 := \langle G_{\mdeg(K)}(K)\rangle.
\]
By additivity of the min-degree ($\mdeg(JK) = \mdeg(J)+\mdeg(K)$), and since any product involving a generator of $J$ or $K$ of degree strictly larger than the minimal one has degree strictly larger than $\mdeg(JK)$, one obtains
\begin{equation}\label{eq: factor equigen}
        \langle G_{\mdeg(JK)}(JK)\rangle = K_1 \cdot K_2.
    \end{equation}
Moreover, $G_{\mdeg(J)}(J)$ is a subset of the antichain $G(J)$, hence is itself an antichain, and therefore coincides with $G(K_1)$ (similarly, $G_{\mdeg(K)}(K) = G(K_2)$). Thus, if $J,K\in \mathcal N(R)$, i.e. $\gcd(G(K_1)) = \gcd(G_{\mdeg(J)}(J)) = 1$ and $\gcd(G(K_2)) = \gcd(G_{\mdeg(K)}(K)) = 1$, by applying Lemma~\ref{lem: gcd IJ} to the pair $K_1,K_2$ we get:
\begin{equation*}
\gcd(G(K_1K_2)) = \gcd(G(K_1)) \cdot \gcd(G(K_2)) = 1,
\end{equation*}
i.e., $JK\in \mathcal{N}(R)$ as $G(K_1K_2)=G_{\mdeg(JK)}(JK)$ by \eqref{eq: factor equigen}. Thus $\mathcal N(R)$ is a submonoid of $\Mon_1(R)$. Moreover, it is divisor-closed in $\Mon(R)$. Indeed, if $JK\in\mathcal N(R)$, then necessarily also $\gcd(G_{\mdeg(J)}(J)) = \gcd(G_{\mdeg(K)}(K)) = 1$, again by \eqref{eq: factor equigen} and Lemma \ref{lem: gcd IJ}.

\medskip

Finally, observe that $\mathcal E(R)\subseteq\mathcal N(R)$. In fact, every equigenerated ideal $I\in\Mon_1(R)$ trivially satisfies $\gcd(G(I))=\gcd(G_{\mdeg(I)}(I))=1$, as $G(I)=G_{\mdeg(I)}(I)$. Thus, such ideals belong to $\mathcal N(R)$, which we have proved to be divisor-closed. Since $\mathcal E(R)$ is the smallest divisor-closed submonoid of $\Mon(R)$ containing them, clearly $\mathcal E(R)\subseteq\mathcal N(R)$.
\end{remark}

From now on, we will focus on the bivariate case, as for $R=K[X,Y]$ we have a very concrete description of $\mathcal E(R)$.

\begin{proposition}\label{prop: mon equigen bivar}
    Let $R=K[X,Y]$. Then we have
    \[ \mathcal E(R) = \{ I\in\Mon(R)\colon X^{\mdeg(I)},Y^{\mdeg(I)}\in G(I) \}. \]
\end{proposition}
\begin{proof}
    Note that the set on the right-hand side is precisely the submonoid $\mathcal N(R)\subseteq \Mon_1(R)$ considered in Remark \ref{rmk: mon equigen ez}. Indeed, in two variables, $\gcd(G_m(I))=1$ if and only if $G_m(I)$ contains both $X^m$ and $Y^m$. Whence, we already know that $\mathcal E(R)\subseteq\mathcal N(R)$. 

    Let $I\in\mathcal N(R)\setminus\{R\}$ and $m=\mdeg(I)\ge 1$. To show that $I\in\mathcal E(R)$, it suffices to find a $J\in\Mon(R)$ such that $IJ$ is equigenerated and $\gcd(G(IJ))=1$. Set  $J=\langle X,Y\rangle^{m-1}=\langle X^{m-1-k}Y^k\colon k\in [0,m-1]\rangle$; we claim that $IJ=\langle X,Y\rangle^{2m-1}$. Indeed, for every monomial $X^{2m-1-k}Y^k$ with $k\in [0,2m-1]$, either $2m-1-k\ge m$ or $k\ge m$;
    as $X^m,Y^m\in G(I)$ and $G(J)$ contains all monomials 
    of degree $m-1$, this shows that $\langle X,Y\rangle^{2m-1}\subseteq IJ$. Moreover, as $\mdeg(IJ)=\mdeg(I)+\mdeg(J)=2m-1$, every monomial belonging to $IJ$ has degree at least $2m-1$, and therefore belongs to $\langle X,Y\rangle^{2m-1}$. A monomial in $IJ$ of degree $2m-1$ is actually an element of $G(IJ)$. Thus,
    we conclude that $IJ=\langle X,Y\rangle^{2m-1}$. Finally, $\langle X,Y\rangle^{2m-1}$ is equigenerated and $\gcd\bigl(G(\langle X,Y\rangle^{2m-1})\bigr)=1,$ since both $X^{2m-1}$ and $Y^{2m-1}$ are minimal generators. Consequently, $IJ\in\mathcal E(R)$, and hence $I\in\mathcal E(R)$.
\end{proof}

\begin{remark}\label{rmk: diagonal atoms}
    If $I\in\mathcal E(R)$ and $I=JK$ for some $J,K\in\Mon(R)\setminus\{R\}$, then by Proposition \ref{prop: mon equigen bivar} and the fact that $\mathcal E(R)$ is divisor-closed in $\Mon(R)$, we have $X^a,Y^a\in G(J)$ and $X^b,Y^b\in G(K)$, where $a=\mdeg(J)\ge 1$ and $b=\mdeg(K)\ge 1$. Thus, $X^aY^b,X^bY^a\in I$, and since they have degree
$\mdeg(I)=a+b$, they belong to $G(I)$. This observation provides a simple sufficient condition for an element of $\mathcal E(R)$ to be an atom of $\Mon(R)$. Let $I\in\mathcal E(R)$ with $\mdeg(I)=m$ be such that, for every $k\in [1,m-1]$, at least one of the monomials $X^k Y^{m-k},X^{m-k}Y^k$ does \emph{not} belong to $G(I)$; then $I\in\mathscr A(\Mon(R))$. In particular, if $G_m(I) = \{X^m,Y^m\}$, then $I$ is an atom.
\end{remark}

We want to compute how many atoms there are in $\mathcal E(R)$, in the following sense. For a given $m\in\mathbb N^+$, denote by $\mathfrak T_m$ and $\mathfrak T_{\le m}$ the set of all ideals in $\mathcal E(R)\setminus\{R\}$ of min-degree, respectively, equal to $m$ and at most $m$. Furthermore, set 
\[\mathfrak A_m:=\mathscr A(\Mon(R))\cap\mathfrak T_m \quad\text{ and }\quad \mathfrak A_{\le m}:=\mathscr A(\Mon(R))\cap\mathfrak T_{\le m}.\] 
Our objective is to estimate the quotient \[\delta_m:={\lvert \mathfrak A_{\le m} \rvert}/{\lvert \mathfrak T_{\le m} \rvert},\] in particular for large $m$.

\begin{lemma}\label{lem: densita equigen}
    For each $m\in\mathbb N^+$, we have
    \[ \lvert \mathfrak T_m \rvert = \frac{1}{m+1}\binom{2m}{m},\qquad\lvert \mathfrak A_m \rvert \ge \frac{1}{m}\binom{2(m-1)}{m-1}. \]
\end{lemma}

\begin{proof}
    Fix $m\ge 1$. Note that the ideals in $\mathfrak T_m$ bijectively correspond to the sets of exponents of their minimal generators, via the map
    \[ \mathfrak T_m\ni I \mapsto \Gamma(I) := \{(a,b)\colon X^aY^b\in G(I)\}\subseteq \mathbb N^2. \]
By Proposition \ref{prop: mon equigen bivar}, both $X^m$ and $Y^m$ belong to $G(I)$. We may therefore order the points of $\Gamma(I)$ by increasing second component as
    \[\Gamma(I)=\{(a_1,b_1),(a_2,b_2),\ldots,(a_r,b_r)\},\]
    where
    \[
        r=\mu(I), \qquad(a_1,b_1)=(m,0),
        \qquad
        (a_r,b_r)=(0,m),
    \]
    and
    \[
        a_1>a_2>\cdots>a_r,
        \qquad
        b_1<b_2<\cdots<b_r.
    \]
    Indeed, these strict inequalities follow from the fact that
    $G(I)$ is an antichain.

    For each $k\in[1,r-1]$, join $(a_k,b_k)$ first to
    $(a_k,b_{k+1})$ by a vertical segment and then
    $(a_k,b_{k+1})$ to $(a_{k+1},b_{k+1})$ by a horizontal segment.
    This produces a lattice path from $(m,0)$ to $(0,m)$ consisting of
    vertical and leftward horizontal steps. Since
    \[        a_k+b_k\ge m
        \qquad\text{for every }k\in[1,r],
    \]
    the path stays above the line segment joining $(m,0)$ and $(0,m)$, whose equation is $a+b=m$. Thus, up to a $90^{\circ}$ rotation, this is what is usually called a \emph{Dyck path of semilength $m$} \cite[Chapter 1.5]{StaCat}.

    Conversely, let $P$ be any lattice path from $(m,0)$ to $(0,m)$ consisting of upward vertical and leftward horizontal steps and staying in the region $a+b\ge m$. Besides the two endpoints, consider the vertices of $P$ at which a horizontal segment is followed by a vertical segment, and let $\Gamma(P)$ be the set consisting of these vertices together with the two endpoints. Their first coordinates strictly decrease and their second coordinates strictly increase, so the corresponding monomials form an antichain and they are precisely the minimal generators of $I_P:=\langle X^aY^b:(a,b)\in\Gamma(P)\rangle$. Moreover, every point of $\Gamma(P)$ satisfies $a+b\ge m$, while the endpoints $(m,0)$ and $(0,m)$ have degree $m$. Hence $\mdeg(I_P)=m$. Finally, since $X^m,Y^m\in G(I_P)$, Proposition~\ref{prop: mon equigen bivar} gives $I_P\in\mathcal E(R)$. Therefore $I_P\in\mathfrak T_m$ and the two constructions are inverse to each other.

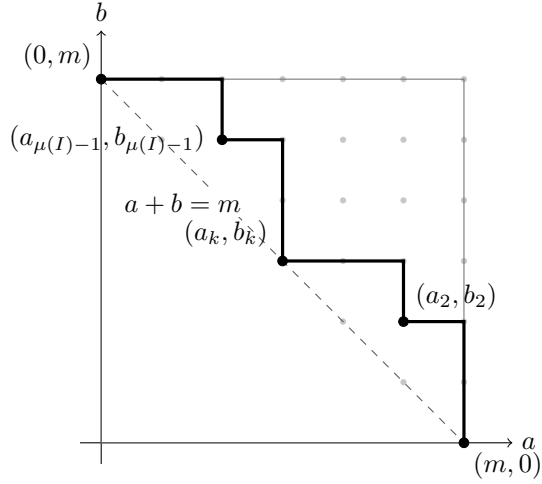
\begin{figure}[ht]
\centering
\begin{tikzpicture}[scale=0.8]
  \def\m{6}

  \foreach \a in {0,...,6}{
    \foreach \b in {0,...,6}{
      \pgfmathtruncatemacro{\s}{\a+\b}
      \ifnum\s<6\relax
      \else
        \fill[black!22] (\a,\b) circle (1.45pt);
      \fi
    }
  }

  \draw[->] (-0.35,0) -- (6.8,0) node[right] {$a$};
  \draw[->] (0,-0.35) -- (0,6.8) node[above] {$b$};
  \draw[thin, black!55] (0,0) rectangle (6,6);

  \draw[dashed, black!65] (0,6) -- (6,0);
  \node[fill=white, inner sep=1pt] at (1.35,3.95) {$a+b=m$};

  \foreach \a/\b in {6/0,5/2,3/3,2/5,0/6}{
    \filldraw[black] (\a,\b) circle (2.25pt);
  }

  \draw[very thick]
    (6,0) -- (6,2) -- (5,2) -- (5,3)
    -- (3,3) -- (3,5) -- (2,5) -- (2,6) -- (0,6);

  \node[below right] at (6,0) {$(m,0)$};
  \node[above left] at (0,6) {$(0,m)$};
  \node[above right=1pt] at (5,2) {$(a_2,b_2)$};
  \node[above left=2pt] at (3,3) {$(a_k,b_k)$};
  \node[left=3pt] at (2,5) {$(a_{\mu(I)-1},b_{\mu(I)-1})$};
\end{tikzpicture}
\caption{The gray dots are the lattice points $(a,b)\in[0,m]^2$ satisfying
$a+b\ge m$, through which a Dyck path may pass. The black dots form one particular set $\Gamma(I)$: besides the endpoints, they are precisely
the vertices at which a horizontal segment is followed by a vertical
one. The thick line is its associated Dyck path.}
\end{figure}

Hence, the elements of $\mathfrak T_m$ are in bijection with the Dyck paths of semilength $m$. The number of such paths, and thus $\lvert \mathfrak T_m \rvert$, equals the $m$-th \emph{Catalan number} (cf. \cite[Theorem 1.5.1(vi)]{StaCat} or \cite[Exercise 6.19(i)]{Stanley}), i.e.,
    \[ \lvert \mathfrak T_m \rvert = C_m := \frac{1}{m+1}\binom{2m}{m}. \]

   In order to bound $\lvert \mathfrak A_m \rvert$ from below, we use Remark \ref{rmk: diagonal atoms}: if $I\in\mathfrak T_m$ is such that $G_m(I)=\{X^m,Y^m\}$, then $I\in\mathfrak A_m$. The Dyck paths corresponding to such ideals are those that touch the diagonal from $(m,0)$ to $(0,m)$ only at the endpoints. 

    Indeed, any interior point at which the path meets the diagonal must be reached by a horizontal step and followed by a vertical step, since the path is contained in the region $a+b\ge m$. Hence it belongs to $\Gamma(P)$ and corresponds to a minimal generator of degree $m$. 
   
   For $m\ge 2$, these paths in turn correspond (by cutting the segments from $(m,0)$ to $(m,1)$ and from $(1,m)$ to $(0,m)$) to lattice paths from $(m,1)$ to $(1,m)$ staying above the line segment between these two points, i.e., Dyck paths of semilength $m-1$. Thus, we obtain
    \[ \lvert \mathfrak A_m \rvert \ge C_{m-1} = \frac{1}{m}\binom{2(m-1)}{m-1}.\]
    The above estimate also holds for $m=1$, as $\langle X,Y\rangle\in\mathfrak A_1$ and so $\lvert \mathfrak A_1 \rvert\ge C_0=1$.
\end{proof}

\begin{proposition}
    For every $m\in\mathbb N^+$, we have $\delta_m\ge 1/4$.
\end{proposition}
\begin{proof}
    By Lemma \ref{lem: densita equigen}, for every $m\ge 1$ we have
    \[ \lvert \mathfrak T_m \rvert = C_m,\qquad\lvert \mathfrak A_m \rvert \ge C_{m-1}, \]
    where $C_m=\frac{1}{m+1}\binom{2m}{m}$. It follows that, for every $m\ge 1$,
    \[ \lvert \mathfrak A_m \rvert \ge \frac{(2m-2)!}{m(m-1)!(m-1)!} = \frac{m(m+1)}{2m(2m-1)}\cdot\frac{(2m)!}{(m+1)(m!)^2} = \frac{m+1}{2(2m-1)}\lvert \mathfrak T_m \rvert \ge \frac{\lvert \mathfrak T_m \rvert}{4}. \]
    Hence,
    \[ \lvert \mathfrak A_{\le m}\rvert = \sum_{k=1}^m \lvert \mathfrak A_k\rvert \ge \frac{1}{4}\sum_{k=1}^m \lvert \mathfrak T_k \rvert = \frac{\lvert \mathfrak T_{\le m} \rvert}{4},\]
    and therefore
    \[ \delta_m = \frac{\lvert \mathfrak A_{\le m} \rvert}{\lvert \mathfrak T_{\le m} \rvert} \ge \frac{1}{4}. \]
\end{proof}

\section*{Declaration of generative AI and AI-assisted technologies in the manuscript preparation process}
Figure~1 was constructed with the assistance of Claude Free. ChatGPT (GPT-5.6 Sol) was used for language polishing, improving the presentation of the manuscript, and assisting with bibliographic searches. All AI-assisted content was reviewed and verified by the authors. The authors take full responsibility for the mathematical results, arguments, conclusions, and overall content of the manuscript.

\end{document}